\documentclass{article}
\usepackage[T1]{fontenc}
\usepackage{amsfonts,amssymb,amsmath,color,amsthm,float}
\newcommand{\bN}{\mathbb{N}}
\newcommand{\bZ}{\mathbb{Z}}
\newcommand{\bC}{\mathbb{C}}

\newcommand{\ini}{\mathrm{in\,}}
\newcommand{\lt}{\mathrm{lt}}
\newtheorem{Theorem}{Theorem}
\newtheorem{Lemma}{Lemma}
\newtheorem{ex}{Example}
\newtheorem{Corollary}{Corollary}
\newtheorem{Claim}{Claim}
\newtheorem{Remark}{Remark} 

\title{Non-degeneracy conditions for plane branches}
\author{Beata Gryszka and Janusz Gwo\'zdziewicz}
\begin{document}
\maketitle

\begin{abstract}
Let $x=t^n$, $y=\sum_{i=1}^{\infty}a_it^i$ be a 
parametrisation of the germ of a complex plane analytic curve $\Gamma$ at the origin. 
Then $\Gamma$ has the implicit equation $f(x,y)=0$ in the neighbourhood of the origin,  
where $f=\sum c_{ij}x^iy^j$ is a Weierstrass polynomial in $\mathbb{C}[[x]][y]$ of degree $n$. Every polynomial depending on coefficients $c_{ij}$ can be expressed 
as a polynomial depending on the coefficients $a_i$. 
We show a partial converse of this property.
\end{abstract}

\section{Introduction}

Let $\alpha=\sum_{i=1}^{\infty}a_i T^i\in\bC[[T]]$ be a formal power series. 
By the {\it support} of $\alpha$ we mean 
the set ${\rm supp}(\alpha)=\{ i\in\bN : a_i\neq 0\}$.
We will call a pair $(T^n,\alpha(T))$ a {\it Puiseux paramertisation} if $\gcd(\{n\} \cup {\rm supp}(\alpha))=1$. 
The {\it characteristic} of a Puiseux parametrisation $(T^n,\alpha(T))$ is the unique sequence $(b_0,b_1,\dots,b_m)$ of positive integers satisfying: 
\begin{itemize}
    \item[(i)] $b_0=n$,
    \item[(ii)] $b_{k+1}=\min ( {\rm supp}(\alpha) \setminus \gcd(b_0,\dots,b_k)\bN)$
    \item[(iii)]  $\gcd(b_0,\dots,b_m)=1$.
\end{itemize}
We also set by convention $b_{m+1}:=+\infty$. We put $e_k:=\gcd(b_0,\dots,b_k)$ for $k=0,\dots,m$ 
and $n_k:=\frac{e_{k-1}}{e_k}$ for $k=1,\dots,m$.

Observe that $e_0>e_1>\cdots>e_m=1$. For every positive integer $k$ we define  $$U_k:=\{z\in\mathbb{C}:z^k=1\}.$$ Then 
$$U_{e_m}\subset U_{e_{m-1}}\subset \cdots \subset U_{e_1}\subset U_{e_0}.$$

\medskip
Let $f\in\bC[[X,Y]]$  
be an irreducible formal power series which is a Weierstrass polynomial of degree $n$, i.e.
$$ f(X,Y)= Y^n + c_1(X)Y^{n-1}+\cdots+c_n(X) $$
where $c_i(X)\in\bC[[X]]$ are power series without constant terms.
Then by the~Puiseux theorem there exists a Puiseux parametrisation  $(T^n,\alpha(T))$
such that $f(T^n,\alpha(T))=0$.  We call $(T^n,\alpha(T))$
a~{\it Puiseux parametrisation of a curve $f(x,y)=0$}.   

Conversely, let $(T^n,\alpha(T))$ be a Puiseux parametrisation.  Then 
$\{(T^n,\alpha(\varepsilon T)): \varepsilon\in U_n\}$
is the set of Puiseux parametrisations of the curve $f(x,y)=0$,
where the irreducible Weierstrass polynomial $f\in\bC[[X,Y]]$ of
degree $n$ is given by 
\begin{equation}\label{equation2} 
f(T^n,Y)=\prod_{\varepsilon\in U_n}(Y-\alpha(\varepsilon T)). \end{equation}

Equation (\ref{equation2})  establishes an $n$-to-one 
correspondence between 
Puiseux para\-metrisations $(T^n,\sum_{i=1}^{\infty}a_i T^i)$
and irreducible Weierstrass polynomials 
of degree~$n$. 
The characteristic of an irreducible Weierstrass polynomial $f$ 
is by definition 
the characteristic of any Puiseux parametrisation of $f(x,y)=0$.
The set of all irreducible Weierstass polynomials in $\bC[[X,Y]]$ that have the characteristic $(b_0,b_1,\dots,b_m)$ will be denoted by $\mathcal{F}(b_0,b_1,\dots,b_m)$.

Assume that $f=\sum_{ij} c_{ij} X^iY^j\in\mathbb{C}[[X,Y]]$
is an irreducible power series such that there exists $n>0$ with the property $c_{0j}=0$ for $j<n$ and $c_{0n}\ne 0$. By the Weierstrass Preparation Theorem (see e.g. \cite{hefez}) there exist unique  $u_f,w_f\in\mathbb{C}[[X,Y]]$ such that $f=u_fw_f$, $u_f$ is invertible and $w_f$ is a Weierstrass polynomial of degree $n$.  We  denote by $\widehat{\mathcal{F}}(n,b_1,\dots,b_m)$ 
the set of all irreducible power series $f\in\bC[[X,Y]]$ such that 
the Weierstrass polynomial $w_f$ of $f$ exists and belongs to $\mathcal{F}(n,b_1,\dots,b_m)$. 
Since $f$ and $w_f$ define the same curve,  we can say about  a~Puiseux parametrisation of a curve  $f(x,y)=0$.

Recall the classical Zariski result (see e.g. \cite[p. 995]{z3}, \cite[Theorem 3.3]{Z}, \cite{hh}).  
Suppose that irreducible power series $f_1,f_2\in\bC[[X,Y]]$ are convergent in the neighbourhood of the origin. If $f_1$, $f_2$ have the same characteristic, then the germs of analytic curves 
$\gamma_i=\{(x,y)\in\bC^2: f_i(x,y)=0\}$ at the origin are equisingular, i.e. there exists a homeomorphism $\Phi:U\to V$ of neighbourhoods $U$, $V$
of $(0,0)$ in $\bC^2$ such that $\Phi(\gamma_1\cap U)=\gamma_2\cap V$. 
Conversely, if the germs of analytic curves $xf_1(x,y)=0$, $xf_2(x,y)=0$ are equisingular then $f_1$ and $f_2$ have the same characteristic.  

In view of  Zariski result the set $\widehat{\mathcal{F}}(b_0,b_1,\dots,b_m)$ can be seen as a given equisingularity class of irreducible power series. 
The main objective of this article is to relate polynomials depending on coefficients 
of a Weierstrass polynomial in $\mathcal{F}(n,b_1,\dots,b_m)$ with polynomials 
depending on coefficients of its Puiseux parametrisation.

Let us start with the following example.



\begin{ex}\label{ex:1}
Consider a Puiseux parametrisation $(T^3,aT^4+bT^7+cT^{10})$, 
$a\neq0$ of the curve $f(x,y)=0$.
Then the Weierstrass polynomial $f\in\mathcal{F}(3,4)$ equals
$$
Y^3-(c_4X^{4}+c_5X^{5}+c_6X^{6}+c_7X^{7}+c_8X^{8}+c_9X^{9}+c_{10}X^{10}),$$
where
$c_{4}=a^3$,
$c_{5}=3a^2b$,
$c_{6}=3ab^2+3a^2c$,
$c_{7}=b^3+6abc$,
$c_{8}=3b^2c+3ac^2$,
$c_{9}=3bc^2$,
and $c_{10}=c^3$.

Every polynomial depending 
on the coefficients of $f$ can be expressed as a polynomial from the ring $\bC[a,b,c]$. Since $c_4,\dots,c_{10}$ are homogeneous polynomials of degree three, every such polynomial $Q$ satisfies the condition: \\
$Q(\varepsilon^4 a,\varepsilon^7 b,\varepsilon^{10}c)=Q(a,b,c)$ for $\varepsilon\in U_3$. This condition will be called $U_3$--invariance of $Q$.  The~natural question is whether  
every polynomial from $\bC[a,b,c]$ that satisfies the above condition 
can be represented as a polynomial in $c_4,\dots,c_{10}$. 

Consider $Q=ab^2$. 
One can check using a symbolic algebra system (we used SINGULAR) that $Q$ does not belong to
the ideal generated by $c_4,\dots,c_{10}$ in the ring $\mathbb{C}[a,b,c]$ and thus $Q\notin \mathbb{C}[c_4,\dots,c_{10}]$.
However
$$Q=\tfrac{(3a^2b)^2}{9a^3}=\tfrac{c_5^2}{9c_4}$$ 
is a rational function of $c_4$ and $c_5$. 
\end{ex}

Our main results: 
Theorem~\ref{main},
Corollary~\ref{c:series}, 
and Theorem~\ref{t3}
explain and generalize the effect observed in Example~\ref{ex:1}.

\medskip
The organization of the paper is as follows: 

In Section~\ref{Section:Hensel} we consider a factorization $f=gh$ of 
a Weierstrass polynomial~$f$ to a product of Weierstrass polynomials $g$, $h$. We prove that if for some weight vector $\omega$
the initial weighted polynomials $\ini_{\omega}h$ and $\ini_{\omega}g$
are coprime Weierstrass polynomials, then the coefficients of $g$ and
$h$ depend polynomially on the reciprocal of the resultant 
$Res(\ini_{\omega}g,\ini_{\omega}h)$ 
and coefficients of: $\ini_{\omega}h$, $\ini_{\omega}g$, and~$f$. 
This result is used in the proof of Lemma~\ref{lemma1}.  

In Section~\ref{sect:coeff} we characterise the Puiseux parametrisation of a plane branch curve from a given equisingularity class.
Then we introduce a class of polynomials depending on the coefficients of a Puiseux 
parametrisation that have the same value for 
equivalent parametrisations. We call this class
$U_n$--invariant polynomials. 

In the next section we prove that every $U_n$--invariant polynomial dependent on the coefficients of a Puiseux parametrisation of a plane branch can be expressed as a rational function of the coefficients of the Weierstrass polynomial, which defines this branch.

In Section 5 we reformulate the above result for a power series that defines a given plane branch. 

In the last section we show that some non-degeneracy condition formulated by a polynomial of coefficients of a Puiseux parametrisation of a~plane branch can be expressed as a polynomial condition dependent on the coefficients of a power series associated with this curve. 

\section{The effective Hensel's lemma}\label{Section:Hensel}
Consider a non-zero vector $\omega=(a,b)\in\mathbb{Z}_{\geq  0}^2.$  For a monomial  $x^iy^j$ we define the {\it weight} of this monomial as  the positive integer $\omega(x^iy^j):=ai+bj$. We also say that  $\omega$ is a {\it weight} in the ring $\mathbb{C}[x,y]$. Then $$\mathbb{C}[x,y]=\bigoplus_{t\in\mathbb{Z}_{\geq 0}}R_t$$ is a graded ring, where $R_t$ is spanned by all monomials of weight $t$.

Below we verify that a division with reminder of quasi-homogeneous polynomials leads to 
quasi-homogeneous polynomials. 

\begin{Lemma}\label{division}
Let $E$ be an integer domain and let $f\in R_s$, $g\in R_t$ 
be quasi-homogeneous polynomials in $E[x,y]$. 
Assume that $g$ is monic with respect to $y$.
Then there exist $q, r\in E[x,y]$ such that $qg, r \in R_s$ and 
$$ f=qg+r, \quad \deg_y r<\deg_y g. $$
\end{Lemma}

\begin{proof}
If $\deg_y f<\deg_yg$ then there is nothing to prove since in this case $r=f$ and $q=0$. 
In what follows assume $m:=\deg_yg\leq\deg_y f$ and apply the standard polynomial division algorithm 
to $f_0:=f$ and $g$; 
$$ f_i := f_{i-1}- \tfrac{\lt(f_{i-1})}{y^m} g, \quad\mbox{for } i=1,2,\dots, j $$
performed while $\deg_y f_{i-1} \geq \deg_y g$, where $\lt(f_{i-1})$ denotes the leading term of a polynomial $f_{i-1}$ with respect to $y$.  
At every step of this algorithm $f_i\in R_s$.  Hence $r:=f_j$ and 
$q:=\lt(f_0)/{y^m}+\cdots + \lt(f_{j-1})/y^m$ 
satisfy the statement of the lemma. 
\end{proof}

Let $g,h\in E[y]$, 
$g(y)=g_0+g_1y+\cdots+g_my^m$, 
$h(y)=h_0+h_1y+\cdots+h_ny^n$ be two 
polynomials with coefficient in an integral domain $E$. 
Recall that the determinant of their Sylvester matrix 
$$ S=
\left( \begin{array}{llllll}
g_m & g_{m-1} & \dots       & g_0 & &   \\
       & g_m        & g_{m-1} & \dots & g_0 &   \\
&  \;\;\;\;\ddots &&&& \\
 & & g_m        & g_{m-1} & \dots & g_0   \\
h_n & h_{n-1} & \dots & h_0 & &  \\
       & h_n        & h_{n-1} & \dots & h_0 &  \\
&  \;\;\;\;\ddots &&&& \\
  & & h_n        & h_{n-1} & \dots & h_0 \\
\end{array} \right)
$$
is called the resultant ${\rm Res}(g,h)$ of $g$ and $h$.  The Sylvester matrix has the property 
that for any polynomials 
$a(y)=a_0+a_1y+\cdots+a_{n-1}y^{n-1}$, 
$b(y)=b_0+b_1y+\cdots+b_{m-1}y^{m-1}$, and 
$f(y)=f_0+f_1y+\cdots+f_{m+n-1}y^{m+n-1}$ 
with coefficients in $E$ 
$$f=ga+hb \Leftrightarrow (f_{m+n-1},\dots, f_0)=(a_{n-1},\dots,a_0,b_{m-1},\dots,b_0)\cdot S. $$
Multiplying the above equation by the adjugate $adj(S)$ of $S$, and using Cramer's rule 
$S\cdot adj(S) = \det(S) I$ we get
$$ (f_{m+n-1},\dots, f_0)\cdot adj(S)=\det(S) (a_{n-1},\dots,a_0,b_{m-1},\dots,b_0).  
$$
In consequence we have the following property

\begin{Lemma}\label{resultant0}
Let $g(y)=g_0+\cdots+g_my^m$, 
$h(y)=h_0+\cdots+h_ny^n$, 
$f(y)=f_0+\cdots+f_{m+n-1}y^{m+m-1}$ be polynomials in variable $y$ with 
coefficients treated as indeterminants. Then there exist unique polynomials 
$A_0,\dots,A_{n-1},B_0\dots,B_{m-1}$ in the ring $\bZ[g_0,\dots,g_m,h_0,\dots,h_n,f_0\dots,f_{m+n-1}]$ such that for 
$a(y)=\sum_{i=0}^{n-1}A_iy^i$, $b(y)=\sum_{i=0}^{m-1}B_i y^i$ we have 
$$ ga+hb=Res(g,h) f. $$
\end{Lemma}

\begin{Lemma}\label{resultant}
Let $g\in R_s$, $h\in R_t$ be polynomials coprime in the ring $\bC[x,y]$ 
and monic with respect to $y$. 
Then for every non-negative integer $i$ and for every $f\in R_{s+t+i}$ 
such that $\deg_y f<\deg_y gh$ there exist unique 
polynomials $\phi\in R_{t+i}$ and $\psi\in R_{s+i}$ such that
\begin{equation}
f=g\phi + h\psi, \quad \deg_y\psi<\deg_y g, \quad \deg_y\phi<\deg_y h .
\end{equation}
Moreover the coefficients of $\phi$ and $\psi$ are of the form $[\phi]_{ij}=R^{-1}\Phi_{ij}$, 
$[\psi]_{ij}=R^{-1}\Psi_{ij}$ where 
$R:={\rm  Res}_y(g,h)|_{x=1}$ and $\Phi_{ij}$, $\Psi_{ij}$ are polynomials depending on: 
coefficients of~$f$, coefficients of~$g$, and coefficients of $h$. 
\end{Lemma}

\begin{proof}
By \cite[Lemma~2.5]{proc} there exist polynomials $a\in R_{t+i}$ and $b\in R_{s+i}$ such that
$f=ga + hb$.  Let us perform a polynomial division $a={q}h+r$ and let $\phi=r$,  
$\psi={q}g+b$. Then $\deg_y \phi <\deg_y h$ and 
$$ f=ga + hb=g({q}h+r)+hb=gr+h({q}g+b)=g\phi + h\psi. $$
 
By Lemma~\ref{division} we have $\phi\in R_{t+i}$, hence 
$h\psi=f-g\phi \in R_{s+t+i}$ which implies that $\psi \in R_{s+i}$.  
Moreover $\deg_y h\psi \leq \max(\deg_y f, \deg_y g\phi)<\deg_y gh$ which gives
$\deg_y \psi <\deg_y g$.

\medskip
Now we will prove uniqueness. 
Suppose that $\phi_1\in R_{t+i}$, $\psi_1\in R_{s+i}$ are polynomials such that 
$f=g\phi_1 + h\psi_1$, $\deg_y\psi_1<\deg_y g$, $\deg_y\phi_1<\deg_y h$.
Then $g(\phi_1-\phi)=h(\psi-\psi_1)$. Since $g$ and $h$ are coprime, $g$ divides $\psi-\psi_1$ which 
together with inequality $\deg_y(\psi-\psi_1)<\deg_y g$ gives $\psi-\psi_1=0$.
In the same way we can show that $\phi_1=\phi$. 

Let $F=f(1,y)$, $G=g(1,y)$, $H=h(1,y)$, $\Phi=\phi(1,y)$, $\Psi=\psi(1,y)$. Then 
$$  G\Phi+H\Psi=F .
$$

By the assumptions that $g$, $h$  are coprime in $\bC[x,y]$ we get ${\rm Res}_y(g,h)\neq 0$. 
By weighted homogeneous property of resultants ${\rm Res}_y(g,h)$ is a monomial. 
Hence ${\rm Res}(G,H)={\rm Res}_y(g,h)|_{x=1}$ is nonzero.  
Hence, by Lemma~\ref{resultant0} polynomials $\Phi$ and $\Psi$ 
and so $\phi$ and $\psi$ have coefficients described in the second part of the lemma. 
\end{proof}

\begin{Lemma}\label{RSef}
Let $f\in \mathbb{C}[[x]][y]$ be a Weierstrass polynomial and 
let $\omega$ be a weight vector such that $\ini_{\omega} f$ and $f$ 
have the same degree with respect to $y$.
Assume that $\ini_{\omega} f = GH$, where $G$, $H$ 
are Weierstrass polynomials coprime in $\mathbb{C}[x,y]$.
Then there are unique Weierstrass polynomials $g,h\in \mathbb{C}[[x]][y]$
such that $f=gh$, $\ini_{\omega} g = G$, and $\ini_{\omega} h = H$. 

Moreover 
if $f= f_{s+t} + f_{s+t+1} + \cdots$, 
$g= g_{s} + g_{s+1} + \cdots$,
$h= h_{t} + h_{t+1} + \cdots$
are representations of $f$, $g$ $h$ as sums of quasi-homogeneous polynomials
and $R:=Res(G,H)|_{x=1}$,  then 
the coefficients of $R^{2i-1}g_{s+i}$ and $R^{2i-1}h_{t+i}$ are polynomials with integer coefficients depending on:
coefficients of $f_{s+t+k}$ for $1\leq k\leq i$, coefficients of $g_s$, and coefficients of $h_t$ for 
$i=1,2,\dots$.
\end{Lemma}

\begin{proof}
By assumptions  $g_s=G$, $h_t=H$ and $f_{s+t}=g_s h_t$. 
Writing the left hand side of 
$(g_s + g_{s+1} + \cdots)(h_t + h_{t+1} + \cdots)=f_{s+t} + f_{s+t+1} + \cdots$ 
as a sum of quasi-homogeneous polynomials we see that the equality $gh=f$ is equivalent with the system of equations 
$$ \sum_{k+l=i} g_{s+k}h_{t+l} = f_{s+t+i} \quad \mbox{for $i=1,2,\dots$} $$
Using these equations and Lemma \ref{resultant} we can compute recursively polynomials $g_{s+i}$, $h_{t+i}$.  

Assume that polynomials $g_{s+j}$, $h_{t+j}$ are already computed for $1\leq j<i$ 
and satisfy the statement of the lemma. Then 
$$ g_{s}h_{t+i} + g_{s+i}h_{t} = f_{s+t+i}-\sum_{\substack{k+l=i,\\k,l<i}} g_{s+k}h_{t+l}. 
$$
Multiplying the above equation by $R^{2i-2}$ we get that 
$$ g_{s}\cdot R^{2i-2} h_{t+i} + h_{t}\cdot R^{2i-2} g_{s+i} = 
    R^{2i-2} f_{s+t+i}-\sum_{\substack{k+l=i,\\k,l<i}} R^{2k-1}g_{s+k}\cdot R^{2l-1}h_{t+l}
$$
is a quasi-homogeneous polynomial of weight $s+t+i$, degree $<\deg_y g_s h_t$ and 
coefficients depending polynomially on coefficients of $f_{s+t}$, \dots, $f_{s+t+i}$, $g_s$, and $h_t$. Thus by Lemma~\ref{resultant} the coefficients of $R^{2i-1}g_{s+i}$, $R^{2i-1}h_{t+i}$  are polynomials depending on: coefficients of: $f_{s+t+k}$ for $1\leq k\leq i$, coefficients of $g_s$, and coefficients of $h_t$.
\end{proof}

\section{Coefficients of a Puiseux parametrisaton}\label{sect:coeff}

The Puisieux parametrisations 
of curves from a given equisingularity class are characterized by the following lemma.

\begin{Lemma}
Let $(b_0,b_1,\dots,b_m)$ be a sequence of positive integers 
such that $b_0=n$ and $b_1<b_2<\dots<b_m$. 
Set $e_k:={\rm gcd}(b_0,b_1,\dots,b_k)$. 
Assume that $e_0>e_1>\cdots >e_m=1$.
Then a Puiseux parametrisation $(T^n,\alpha(T))$ has characteristic $(b_0,b_1,\dots,b_m)$ if and only if 
$$ \{ b_1,\dots,b_m \} \subset {\rm supp}(\alpha) \subset I,$$
where 
\begin{equation} \label{ind}
I=e_0\bN_{+} \cup \bigcup_{i=1}^m e_i\bN\cap[b_i,+\infty).\end{equation}
\end{Lemma}

\begin{proof}
Assume that $(b_0,b_1,\dots,b_m)$ is the characteristic of $(T^n,\alpha(T))$. 
Then for every $j\in\{1,\dots,m\}$ we have $b_j \in {\rm supp}(\alpha)$ and consequently
$\{b_1,\dots,b_m\}\subset {\rm supp}(\alpha)$. 
Take an arbitrary $i\in {\rm supp}(\alpha)$. 
If $i<b_1$, then $e_0|i$ and thus $i\in e_0\mathbb{N}_+\subset I$. Assume that
$b_j\leq i< b_{j+1}$ for some $j\in\{1,\dots,m\}$. Then $i\in e_j\bN$. 
This implies that $i\in e_j\mathbb{N}\cap[b_j,\infty) \subset I$. 

Assume now that  
$$ \{ b_1,\dots, b_m \} \subset {\rm supp}(\alpha) \subset I$$
and a sequence $(b_0',b_1',\dots,b_t')$ is the characteristic of $(T^n,\alpha(T))$.

We will prove by induction on $k$ for $k\in\{0,\dots, m\}$ that  
$$(b_0',\dots, b_k')=(b_0,\dots, b_k).$$

Clearly $b_0'=b_0=n$.
By inductive assumption 
$(b_0',\dots, b_{k}')=(b_0,\dots, b_{k})$.
Then 
$b_{k+1}'=\min ( {\rm supp}(\alpha) \setminus e_k\bN)=b_{k+1}$.
We proved that $(b_0',\dots b_m')=(b_0,\dots b_m)$.
Since $\gcd(b_0',\dots b_m')=1$, we get the lemma.
\end{proof}

Consider a Puiseux parametrisation $(T^n,\alpha(T))$ and the increasing sequence $i_1<i_2<\cdots$ of all indices from the~set~(\ref{ind}). Let $A:=(A_{i})_{i\in I}$ and $C:=(C_{ij})_{(i,j)\in\mathbb{Z}_{\geq 0}^2}$ be sets  of new  variables. Then we say that  $Q\in\mathbb{C}[C,A]$ is $U_n$--{\it invariant} if
$$ Q(\varepsilon^{i_1} A_{i_1},\varepsilon^{i_2} A_{i_2},\dots) = Q(A_{i_1},A_{i_2},\dots) $$
for every $\varepsilon\in U_n$.

\section{Irreducible Weierstrass polynomials}
Let us start with the following lemma concerning complex roots of unity.

\begin{Lemma}\label{L:roots}
Let $n$, $i$ be integers, $n>0$ and let $d=n/\gcd(n,i)$. Then 
\begin{equation}\label{eq:group}
\{\,\varepsilon^i: \varepsilon\in U_n\,\} = U_d, 
\end{equation}
\begin{equation}\label{eq:roots}
\sum_{\varepsilon\in U_n} \varepsilon^i= 
\begin{cases}
n & \mbox{if $i\equiv 0 \mod n$}, \\ 
0    & \mbox{otherwise.}
\end{cases} .
\end{equation}
\end{Lemma}

\begin{proof}
The set $\{\varepsilon^i: \varepsilon\in U_n\}$ is a cyclic subgroup of a multiplicative 
group $\bC^{*}$ generated by $\varepsilon_0^i$ where $\varepsilon_0$ is a primitive $n$th complex 
root of unity. Since the rank of $\varepsilon_0^i$ in $\bC^{*}$ equals $d$, we get~(\ref{eq:group}).

It follows from~(\ref{eq:group}) that $U_n\ni\varepsilon\longrightarrow \varepsilon^i\in U_d$ 
is a group epimorphism. Thus 
$$ \sum_{\varepsilon\in U_n}\varepsilon^i = \gcd(n,i) \sum_{\omega\in U_d} \omega .
$$
This gives~(\ref{eq:roots}).
\end{proof} 

In this section we deal with irreducible Weierstrass 
polynomials in a given equisingularity class 
$\mathcal{F}(b_0,b_1,\dots,b_m)$.  
Let $f=\sum_{ij}c_{ij}X^iY^j\in\mathbb{C}[[X]][Y]$ be an irreducible Weierstrass polynomial from $\mathcal{F}(b_0,b_1,\dots,b_m)$ 
and $(T^n,\alpha(T))$, where $\alpha=\sum_{i\in \mathbb{Z}_{\geq 0}}a_iT^i$, be a Puiseux parametrisation of $f(x,y)=0$.

 Recall that the support of $\alpha $ is contained in the set $I$ given in (\ref{ind}). In the sequel we treat coefficients of $\alpha$ and $f$ as indeterminants and we introduce   $\alpha^{*}:=\sum_{i\in I}A_iT^i$  and $f^{*}=\sum C_{ij}X^iY^j$, where $A:=(A_{i})_{i\in I}$ and  
$C:=(C_{ij})_{(i,j)\in\mathbb{Z}_{\geq 0}^2}$ are sets  of new  variables. We also denote by $A':=(A_i)_{i\in I, i<b_m}$.

 Let $\varepsilon$ be the $n$th primitive root of unity. Using Vieta's formula and the equaliy (\ref{equation2}) for $j=0,\dots,n-1$ we obtain 
$$\sum_{i=0}^{\infty} c_{ij}T^{ni}=(-1)^{n-j}s_{n-j}(\alpha(T),\alpha(\varepsilon T),\dots, \alpha(\varepsilon^{n-1}T)),$$
where  $s_j\in\mathbb{C}[X_1,\dots,X_n]$ denotes the $j$th elementary symmetric polynomial. This implies that for every $i\in\mathbb{Z}_{\geq 0}$ and $j\in\{0.\dots,n-1\}$ there exists a  polynomial $F_{ij}\in\mathbb{C}[A]$ such that $$c_{ij}=F_{ij}(a_{i_1},a_{i_2},\dots),$$ where $i_1<i_2<\cdots$ is the increasing sequence of all indices from $I$.  
Take any $i,t\in\mathbb{Z}_{\geq 0}$ and $j\in\{0,\dots,n\}$. Then $F_{ij}(\varepsilon^{ti_1}a_{i_1},\varepsilon^{ti_2}a_{i_2},\dots)$ corresponds to $(-1)^{n-j}s_{n-j}(\alpha(\varepsilon^t T),\alpha(\varepsilon^{t+1} T),\dots, \alpha(\varepsilon^{t+n-1} T))$. Since $s_j$ is symmetric, we deduce that $F_{ij}$ is a $U_n$--invariant polynomial. Thus we introduce the following notation: $$F:=(F_{ij})_{(i,j)\in\mathbb{Z}_{\geq 0}^2}.$$

\begin{Remark}
Let $W$ be a polynomial from $\mathbb{C}[C]$. Then after the substitution $C=F$ the polynomial $W$ becomes  $U_n$--invariant.
\end{Remark}

\begin{Lemma}\label{lemma1}
\begin{itemize}
\item[{\rm (i)}] There exist a polynomial $M\in\mathbb{C}[C,A']$ and a monomial $N$ in variables $A_{b_1},\dots,A_{b_{m-1}}$ such that 
\begin{equation}\label{eq:2}
A_{b_m}^{n_m}=\tfrac{M(F,A')}{N}.
\end{equation}
\item[{\rm (ii)}]  For every $i\geq b_m$ there exist 
$W_i\in\bC[C,A',\tfrac{1}{A_{b_1}},\dots,\tfrac{1}{A_{b_{m-1}}}, \tfrac{1}{A_{b_m}^{n_m}}]$
such that after substituting $C=F$ we have  $$A_i = W_i A_{b_m}^{j_i}$$
where $j_i$ is the~unique integer such that $0\leq j_i<n_m$ and $b_m j_i \equiv i \mod n_m$.
\end{itemize}
\end{Lemma}
\begin{proof}
  Let us substitute $C=F$. Consider  $\lambda:=\sum_{i\in I, i<b_m}A_iT^i$ and 
 $$
\widehat{f}:=f^{*}(T^n,Y+\lambda(T))=
\prod_{\varepsilon\in U_n}\Bigl(Y+\lambda(T)-\sum_{i\in I}A_i\varepsilon^iT^i\Bigr).$$
Then  
$$\widehat{f}=\widehat{f}_1\cdots\widehat{f}_m,$$
where 
$$\widehat{f}_j=\prod_{\varepsilon\in U_{e_{j-1}}\setminus U_{e_j}}
\Bigl(Y+\lambda(T)-\sum_{i\in I}A_i\varepsilon^iT^i\Bigr) \ \ \ \ {\rm for }\ \ \  j=1,\dots,m-1,$$
$$\widehat{f}_m=\prod_{\varepsilon\in U_{e_{m-1}}}
\Bigl(Y-\sum_{i=b_m}^{\infty}A_i\varepsilon^i T^i\Bigr) .$$

Computing the weighted initial polynomials we get 
$${\rm in}_{(1,b_j)}\widehat{f}_j = 
Y^{-e_j}((Y+A_{b_j}T^{b_j})^{n_j}-A_{b_j}^{n_j}T^{b_j n_{j}})^{e_j} 
\quad \mbox{for $j=1,\dots,m-1$}, $$
$${\rm in}_{(1,b_m)}\widehat{f}_m=Y^{n_m}-A_{b_m}^{n_m}T^{n_m b_m}.$$

Hence
\begin{eqnarray*}
 \widehat{f}(T,0) &=& \prod_{j=1}^m \widehat{f}_j (T,0) = 
    -A_{b_m}^{n_m}T^{n_m b_m} \prod_{j=1}^{m-1} n_j^{e_j}A_{b_j}^{e_{j-1}-e_j} T^{(e_{j-1}-e_j)b_j} + \mbox{h.o.t.}  \\
  &=& \Bigl(- A_{b_m}^{n_m} \prod_{j=1}^{m-1} n_j^{e_j}A_{b_j}^{e_{j-1}-e_j} \Bigr) T^{\bar b_m}  + \mbox{h.o.t.}
\end{eqnarray*}
where $\bar b_m=e_{m-1}b_m+\sum_{j=1}^{m-1}(e_{j-1}-e_j)b_j$. 
On the other hand it follows from $\widehat{f}(T,0)= f^*(T^n,\lambda(T))$ 
that the coefficient of $\widehat{f}(T,0)$ at $T^{\bar b_m}$ is a polynomial depending 
on $C$ and $A'$.   Thus
\begin{equation}\label{eq:2a}
A_{b_m}^{n_m}=\tfrac{M}{N}
\end{equation}
where $M$ is a polynomial depending on $C$, $A'$ and $N$
is a monomial in variables $A_{b_1},\dots, A_{b_{m-1}}$.

\begin{Claim}
 For every integer $i$ such that  $1\leq i\leq m$ the coefficients of 
$\widehat{f}_i \cdots \widehat{f}_m$  are polynomials in $C,A',\tfrac{1}{A_{b_1}},\dots,\tfrac{1}{A_{b_{i-1}}}$. 

\end{Claim}
\begin{proof}
We will prove this claim by induction on $i$. 

For $m=1$ we have 
$\widehat{f}=\widehat{f}_1$
and there is nothing to prove. 
Thus assume that $m\geq 2$
and consider 
a weight $\omega:=(1,b_i)$. For $1\leq i< m$ we have 
$ 
\widehat{f}_{i}\cdots \widehat{f}_m = 
\widehat{f}_i \cdot (\widehat{f}_{i+1}\cdots \widehat{f}_m)
$, 
$\ini_{\omega}\widehat{f}_i = 
Y^{-e_i}((Y+A_{b_i}T^{b_i})^{n_i}-A_{b_i}^{n_i}T^{b_i n_{i}})^{e_i}$,
$\ini_{\omega}(\widehat{f}_{i+1} \cdots \widehat{f}_m)=Y^{e_i}$,  and 
${\rm Res}(\ini_{\omega}\widehat{f}_i,Y^{e_i})|_{T=1}=(n_i^{e_i}A_{b_i}^{e_{i-1}-e_i})^{e_i}$. 

If $i=1$, then using Lemma \ref{RSef} we obtain that coefficients of $\widehat{f}_{2}\cdots \widehat{f}_m$ are polynomials in $C$, $A'$, $\tfrac{1}{A_{b_1}}$. 

Take an arbitrary 
$i\in\{1,\dots,m-1\}$
and suppose that $\widehat{f}_{i}\cdots \widehat{f}_m$ has coefficients in $C$, $A'$, $\tfrac{1}{A_{b_1}},\dots,\tfrac{1}{A_{b_{i-1}}}$.  Since $\widehat{f}_{i}\cdots \widehat{f}_m=\widehat{f}_i\cdot(\widehat{f}_{i+1}\cdots \widehat{f}_m)$, by inductive assumption and by Lemma \ref{RSef} we have that coefficients of $\widehat{f}_{i+1}\cdots \widehat{f}_m$ are in $\mathbb{C}[C,A',\tfrac{1}{A_{b_1}},\dots,\tfrac{1}{A_{b_i}}]$ and thus the claim is proven.
\end{proof} 

Let $\omega :=(1,b_m)$ and $G:=Y-\hat A_{b_m} T^{b_m}$ be a linear factor of 
$Y^{n_m}-A_{b_m}^{n_m}T^{n_m b_m}$. Then 
$${\rm in}_{\omega}\widehat{f}_m = GH $$ 
where 
$$ H=Y^{n_m-1}+\hat A_{b_m}Y^{n_m-2}T^{b_m}+\cdots+\hat A_{b_m}^{n_m-1}T^{(n_m-1)b_m}.  
$$
Clearly ${\rm Res}(G,H)=n_m\hat A_{b_m}^{n_m-1}T^{(n_m-1)b_m}$.  
It follows from Claim~1 that coefficients of $\widehat{f}_m$ are polynomials in 
$C$, $A'$,  $\tfrac{1}{A_{b_1}},\dots,\tfrac{1}{A_{b_{m-1}}}$.
By Lemma~\ref{RSef} we have $\widehat{f}_m=gh$, where $\ini_{\omega} g =G$, $\ini_{\omega} h =H$ 
and coefficients $\hat A_i$ of 
$$ g=Y-\sum_{i=b_m}^{\infty}\hat A_i T^i $$	
are polynomials in 
$C,A',\tfrac{1}{A_{b_1}},\dots,\tfrac{1}{A_{b_{m-1}}},\tfrac{1}{\hat A_{b_m}^{n_m}},\hat A_{b_m}$.

Using~(\ref{eq:2}) and the identity 
$\hat A_{b_m}^{n_m}=A_{b_m}^{n_m}$ 
we can rewrite $\hat A_i$ as the sum 
\begin{equation}\label{eq:3}
\hat A_i=\sum_{j=0}^{n_m-1} M_{i,j}\hat A_{b_m}^j ,
\end{equation}
where $M_{i,j}$ are {polynomials} in  $C, A',\tfrac{1}{A_{b_1}}\dots,\tfrac{1}{A_{b_{m-1}}}, \tfrac{1}{A_{b_m}^{n_m}}.$ 


For every $\varepsilon\in U_{e_{m-1}}=U_{n_m}$, the polynomial 
$Y-\sum_{i=b_m}^{\infty}\varepsilon^iA_iT^i$
is a factor of $\widehat f_{m}$. Hence substituting in~(\ref{eq:3}) 
$\hat A_{b_m}:=\varepsilon^{b_m } A_{b_m}$ we get by 
$\hat A_i=\varepsilon^i A_i$ identities
$$ \varepsilon^i A_i = \sum_{j=0}^{n_m-1} M_{i,j} \cdot(\varepsilon^{b_m } A_{b_m})^j 
.$$

Writing the above formulas in the form 
$$ A_i = \sum_{j=0}^{n_m-1} \varepsilon^{b_m j-i} M_{i,j} A_{b_m}^j $$
and taking their sum over all $\varepsilon\in U_{n_m}$ we get 
$$ n_m A_i = \sum_{j=0}^{n_m-1} \sum_{\varepsilon\in U_{n_m}}\varepsilon^{b_m j-i} M_{i,j} A_{b_m}^j $$

Thus by Lemma~\ref{L:roots} we have  $A_i = W_i A_{b_m}^{j_i}$,
where $W_i=M_{i,j_i}$ and  $0\leq j_i<n_m$, $b_m j_i \equiv i \mod n_m$.

\end{proof}

\begin{Corollary}\label{c:1}
Let $k\in\{1,\dots,m\}$.
\begin{itemize}
\item[{\rm (i)}] There exist a polynomial $M$ depending  on $C$ and $A_i$ for $i<b_k$ and a monomial $N$
 in variables $A_{b_1},\dots, A_{b_{k-1}}$ such that after substituting $C=F$ we have
\begin{equation}\label{eq:Ab_k}
A_{b_k}^{n_k}=\tfrac{M}{N}.
\end{equation}
Moreover, for every $l\in\mathbb{Z}_{\geq 0}$ there exists an integer $r>l$ such that $M^r$ and $N^r$ are $U_n$--invariant.
\item[{\rm (ii)}] For every integer $i$ such that $b_k\leq i< b_{k+1}$ there exist a polynomial 
$W_i\in\bC[C,\tfrac{1}{A_{b_1}},\dots,\tfrac{1}{A_{b_{k-1}}}, \tfrac{1}{A_{b_k}^{n_k}}][A]$
depending on variables $A_j$ for $j<b_k$ such that after substituting $C=F$ we have
\begin{equation}\label{eq:A_i}
A_i = W_i A_{b_k}^{j_i}
\end{equation}
where $j_i$ is the unique integer such that $0\leq j_i<n_k$ and $b_k j_i \equiv i \mod e_{k-1}$.
\end{itemize}
\end{Corollary}
\begin{proof}
For $k=m$ the statement follows directly from Lemma \ref{lemma1}. 
Take an~arbitrary $k\in\{1,\dots,m-1\}$. 
Let $p:=\sqrt[e_{k}]{f}$ be the $e_{k}$th approximate root of $f$. 
Then $p\in\mathcal{F}(b_0',b_1',\dots,b_k')$, where $$
 b_j'=b_j/e_k \textnormal{\ \ \  for \ \ \ } 0\leq j \leq k$$
and the curve  $p(x,y)=0$ has a Puiseux parametrisation $(T^{n/e_k},\sum_{i=1}^{\infty}d_i T^i)$ such that
$d_i=a_{i e_k}$ for $i\in\{1,\dots,b_{k+1}-1\}$.  Note that for the  sequence $e_0'>e_1'>\cdots>e_k'$ assosciated with the characteristic of $p$ we have $e_j':=n_{j+1}\cdots n_k$ for $j=0,\dots,k-1$ and $e_k'=1$. 

Let $I=\{i_1,i_2,\dots\},$ where $i_1<i_2<\cdots.$ If $I'=\{i_1',i_2',\dots\}$  is the corresponding set of indices for $p$ with the property $i_1'<i_2'<\cdots$, then $$\{i_1',\dots,i_s'\}=\tfrac{1}{e_k}\{i_1,\dots,i_s\},$$ where $i_{s+1}=b_{k+1}$. Consider the sequence of variables  $D:=(D_i)_{i\in I'}$ that corresponds to the sequence of coefficients  of the above Puiseux parametrisation of $p$. We can identify  $$A_j=D_{j/e_k}\textnormal{\ \ \    for\ \ \ }j<b_{k+1}$$ and equivalently we can write $$D_j=A_{je_k}\textnormal{\ \ \ for  \ \ \ }j<b_{k+1}'.$$ In particular $$A_{b_j}=D_{b_j'}\textnormal{\ \ \  for\ \ \  }0\leq j\leq k.$$ 

Let $C'$ be the set of variables that corresponds to the coefficients of $p$ and let $F'$ be the set of polynomials defined similarly as $F$ for $f$. Consider substitutions $C=F$ and $C'=F'$.
 Using Lemma \ref{lemma1} (i) to $p$ we have the equality $$A_{b_k}^{n_k}=D_{b_k'}^{n_k}=\tfrac{M}{N}$$
  for some polynomial $M$ depending on $C'$ and $D_{j}=A_{je_k}$, where $je_k<b_k'e_k=b_k$, and for some monomial $N$ in $D_{b_1'}=A_{b_1},\dots,D_{b_{k-1}'}=A_{b_{k-1}}$. By Lemma \ref{lemma1} (ii) we obtain $$A_i=D_{i/e_{k}}=W_iD_{b_{k}'}^{j_i'}=W_iA_{b_k}^{j_i'}$$ for some polynomial $W_i\in\mathbb{C}[C',\tfrac{1}{A_{b_1}},\dots,\tfrac{1}{A_{b_{k-1}}},\tfrac{1}{A_{b_k}^{n_k}}][A]$ depending on variables $D_j=A_{je_k}$ such that $je_k<b_k$ and for the unique $j_i'\in\{0,\dots,n_k-1\}$    satisfying  the condition $b_k j_i' \equiv i/e_{k} \mod e_{k-1}'$. Since $e_{k-1}'=n_k$, the last congruence implies that $e_{k-1}|b_{k}'j_i'e_{k}-i=b_kj_i'-i$. Thus $j_i'=j_i$. Since coefficients of $p$ depends polynomially on $C$, we obtain (\ref{eq:Ab_k}) and (\ref{eq:A_i}).
  
  Take any $l\in\mathbb{Z}_{\geq 0}$. Let $r$ be an arbitrary multiple of $n$ that is greater than $l$. Then $A_{b_k}^{n_kr}=\tfrac{M^r}{N^r}$ and $N^r$ are $U_n$--invariant  and thus $M^r$ is also   $U_n$--invariant.
\end{proof}

\begin{Lemma}\label{l6} 
Let $k$ be an integer such that $1\leq k\leq m$. 
Assume that $Q\in\mathbb{C}[C,A]$ is a $U_n$--invariant polynomial depending on $A_i$ for $i<b_{k+1}$.  
Then there exist $U_n$--invariant polynomials $V,W\in\mathbb{C}[C,A]$
depending on $A_i$ for $i<b_k$ such  that $Q(F,A) V(F,A)=W(F,A)$. 
Moreover if $A_{b_1},\dots, A_{b_m}\neq 0$, then $V(F,A)\neq0$.
\end{Lemma}
\begin{proof} 
Fix $k$ ($1\leq k\leq m$).  Let $Q\in\mathbb{C}[C,A]$ be a $U_n$--invariant 
polynomial depending on $A_i$ for $i<b_{k+1}$ and write $Q$ as a polynomial in 
$\mathbb{C}[C,A''][A_{i},\dots,A_{i+r}]$, where $A''$ are variables $A_j$ for $j<b_k$ and 
$b_k \leq i < b_{k+1}$. 
Let $\widetilde Q(A_{b_k}):=Q(W_i A_{b_k}^{j_i},\dots,  W_{i+r}A_{b_k}^{j_{i+r}})$.
Then for every $\varepsilon\in U_{e_{k-1}}$ we have 
\begin{align*}
\widetilde Q(\varepsilon^{b_k}A_{b_k})&=
Q(W_i \cdot (\varepsilon^{b_k}A_{b_l})^{j_i},\dots, W_{i+r}\cdot (\varepsilon^{b_k}A_{b_k})^{j_{i+r}})\\ 
 &= Q(\varepsilon^i W_i A_{b_k}^{j_i},\dots, \varepsilon^{i+r}W_{i+r} A_{b_k}^{j_{i+r}})\\
  &=
Q(W_i A_{b_k}^{j_i},\dots, W_{i+r} A_{b_k}^{j_{i+r}}) = \widetilde Q( A_{b_k}).
\end{align*} 

By Lemma~\ref{L:roots} we have $\{\,\varepsilon^{b_k}: \varepsilon\in U_{e_{k-1}}\,\} = U_{n_k}$. 
Hence
\begin{equation}\label{eq:5}
\widetilde Q(\varepsilon A_{b_k})=\widetilde Q( A_{b_k})
\end{equation}
for every $\varepsilon\in U_{n_k}$. 

Let us substitute $C=F$. Using Corollary \ref{c:1} we can write $\widetilde Q(A_{b_k})$ 
as a~Laurent polynomial 
 $$ \widetilde Q(A_{b_k})=\sum Q_i A_{b_k}^i $$
with coefficients $Q_i$ in the ring $\bC[C,A'',\tfrac{1}{A_{b_1}}\dots,\tfrac{1}{A_{b_{k-1}}}].$

By the same argument as in the proof of Lemma~\ref{lemma1} we obtain 
$$ \widetilde Q(A_{b_k})=\sum Q_{n_k i} A_{b_k}^{n_k i} .$$

It follows that $Q$ can be expressed as a polynomial in 
$C,A'',\tfrac{1}{A_{b_1}},\dots,\tfrac{1}{A_{b_{k-1}}}$, $\tfrac{1}{A_{b_k}^{n_k}}$, 
$A_{b_k}^{n_k}$. Using Corollary \ref{c:1} and converting everything to a common denominator, we obtain  a polynomial $M_1$ in variables $C, A''$,  a monomial $N_1$ in variables 
$A_{b_1} ,\dots,A_{b_{k-1}}$  and a non-negative
integer $l$ such that 
$$Q N_1 M^l = M_1.$$
There exist a monomial $N_2$ in variables $A_{b_1} ,\dots, A_{b_{k-1}}$ and $r>l$ 
such that $N_1 N_2$ and $M^r$ are $U_n$--invariant polynomials.  
Hence the polynomials $V= N_1 N_2 M^r$ and $W=M_1N_2 M^{r-l}$ with coefficients in $A'$, $C$
are $U_n$--invariant, $Q V = W$, and $V\neq0$ provided $A_{b_1},\dots, A_{b_m}\neq 0$.

\end{proof}

\begin{Theorem}\label{main}
Let $Q\in\mathbb{C}[A]$ be  $U_n$--invariant.  Then there exist polynomials 
$V,W\in \bC[C]$ such that 
$$ Q=\frac{W(F)}{V(F)}
$$
and $V(F)$ is a monomial in variables $A_{b_1}$, \dots,  $A_{b_m}$.
\end{Theorem}

\begin{proof}
Consider the substitution $C=F$. We will prove by induction on $k\in\{1,\dots, m+1\}$ the following statement:
\begin{itemize}
\item[$(T_k)$] for every $U_n$--invariant polynomial $Q\in\mathbb{C}[C,A]$ that depends on $A_i$ for $i<b_k$
there exist polynomials $M,N\in \bC[C]$ such that $Q=N/M$ and $M\neq 0$ provided $A_{b_1},\dots, A_{b_m}\neq 0$.
\end{itemize}

\medskip
\textit{Proof of $(T_1)$}. 
For indices $i\in I$ such that $i<b_1$ we have $i \equiv 0 \mod n$. Hence the series 
$\sum_{\varepsilon\in U_n} \alpha(\varepsilon T)$ and $n \alpha$ are equal up to order $b_1$.  
By~(\ref{equation2}) the~first of these series is equal to minus coefficient of the polynomial $f(T^n,Y)$ at $Y^{n-1}$. Thus for $i<b_1$ we have 
 $A_i=-\tfrac{1}{n}C_{i/n,n-1}$. Hence $Q$ can be expressed as a polynomial in $C$. 
 
\medskip
\textit{Proof of} $(T_k)\Rightarrow (T_{k+1})$. 
Let $Q$ be a $U_n$--invariant complex polynomial that depends on $A_i$ for $i<b_{k+1}$. 
By Lemma~\ref{l6} there exist $U_n$--invariant polynomials $V,W\in\mathbb{C}[C,A]$ 
depending on $A_i$ for $i<b_k$ such  that $Q V= W$. 
Moreover if $A_{b_1},\dots, A_{b_m}\neq 0$, then $V\neq0$. 

By inductive assumption there are polynomials $M_1,N_1,M_2,N_2\in\bC[C]$ such that 
$V =N_1/M_1$, $W = N_2/M_2$ and  $M_1,M_2\neq 0$ provided $A_{b_1},\dots, A_{b_m}\neq 0$.
Then $Q=\frac{W}{V} = \frac{N_2/M_2}{N_1/M_1}=\frac{N_2 M_1}{N_1 M_2}$. Since $N_1=V M_1$, we get $N_1 M_2 \neq 0$ provided $A_{b_1},\dots, A_{b_m}\neq 0$.

\medskip
We proved by induction the first part of the theorem. The denominator $V$ has the property:
if $A_{b_1},\dots, A_{b_m}\neq 0$, then $V(F)\neq 0$.  Hence by Hilbert Nullstellensatz 
$V(F)$ divides $(A_{b_1}\cdots A_{b_m})^l$ for some positive integer~$l$.
\end{proof}

\section{Irreducible power series}

Let us consider the following variation of the Weierstrass theorem.

\begin{Theorem}[Weierstrass]\label{wei}
Assume that $f =
\sum_{i,j}
c_{ij}X^iY^j \in \mathbb{C}[[X ,Y ]]$
and there exists $m > 0$ such that $c_{0j
} = 0$ for $j < m$ and $c_{0m}=1$. Then
there exist unique $u_f$, $w_f \in\mathbb{C}[[X,Y ]]$ such that $f = u_fw_f $, $u_f (0,0)=1 \ne 0$ and $w_f$
is a~Weierstrass polynomial with respect to $Y$. Moreover, the coefficients of  $w_f$ and $u_f$ depends polynomially on the coefficients of $f$.
\end{Theorem}
\begin{proof}
It is enough to apply the proof of \cite[Theorem 2.3]{hefez} with $r=2$, $X_1=X,X_2=Y$, $G=Y^m$, $Q=u_f^{-1}$, $F=f$ and $R=Y^m-w_f$.
Then $R_{-1}=0$, $c=a_m(0)=1$, $H=H_m=G=Y^m$, $H_{m+i}=0$ for every $i\in\mathbb{Z}_{>0}$, $q_0=1$  and $P=Y^m+(\textnormal{terms in $Y$ of degree }<m)$. Since   $P\in\mathbb{C}[X][Y]$ is monic and on every step of construction we divide polynomial $-q_0F_{m+k}-q_1F_{m+k-1}-\cdots-q_{k-1} F_{m+1}$ by $P$,  the constructed polynomials $q_i$, $R_i$ depends polynomially on the coefficients of~$f$. This implies that coefficents of  $w_f=Y^m-R=Y^m-(R_0+R_1+\cdots)$ and $Q=q_0+q_1+\cdots$ depend polynomially on the coefficients of $f$. Using the formulas from the proof of \cite[Proposition 1.1]{hefez} and the fact that $q_0=1$, we obtain that the coefficients of  $u_f=Q^{-1}$ also depend polynomially on the coefficients of $f$.
\end{proof}

 

 Using Theorem~\ref{main} and   Theorem \ref{wei} we directly obtain the~following fact.
 
\begin{Corollary}\label{c:series}
Let $f\in\widehat{\mathcal{F}}(b_0,b_1,\dots,b_m)$. Assume that $Q$ is a $U_n$--invariant polynomial of the coefficients of  a~Puiseux parametrisation $(T^n,\sum_{i=0}^{\infty}a_iT^i)$ of the curve $f(x,y)=0$. Then $Q$ is a rational function of the coefficients  of $f$ such that the denominator is a monomial in $a_{b_1},\dots,a_{b_m}$.
\end{Corollary}

\section{Non-degeneracy conditions}

Consider 
$f=\sum_{ij} c_{ij} X^iY^j$ which belongs to the equisingularity class 
$\widehat{\mathcal{F}}(n,b_1,\dots,b_m)$ and a Puiseux paramerisation 
$(T^n,\sum_{i\in I} a_i T^i)$ of the curve $f(x,y)=0$.

\begin{Theorem} \label{t3}
Let $\{i_1,\dots, i_s\}$ be a finite subset of $I$. 
Then for every polynomial $Q\in\bC[A_{i_1},\dots, A_{i_s}]$ such that 
$Q(A_{i_1},\dots, A_{i_s})=0 \iff 
Q(\varepsilon^{i_1}A_{i_1},\dots, \varepsilon^{i_s}A_{i_s})=0$
for $\varepsilon\in U_n$
there exits a finite subset $\{\gamma_i,\dots, \gamma_t\}$ of $\bN^2$ 
and a polynomial $W\in\bC[C_{\gamma_1},\dots,C_{\gamma_t}]$
such that 
$Q(a_{i_1},\dots,a_{i_s})=0 \iff W(c_{\gamma_1},\dots,c_{\gamma_t})=0$.
\end{Theorem}

\begin{proof}
Replacing $Q$ by the product 
$\prod_{\epsilon\in U_n} 
Q(\varepsilon^{i_1}A_{i_1},\dots, \varepsilon^{i_s}A_{i_s})$
we may assume without loss of generality that $Q$ is $U_n$--invariant. 
Let $f'=(1/c_{0n})f$. The power series $f'$ satisfies assumptions 
of Corollary~\ref{c:series}. Hence there exists a finite subset 
$\{\gamma_i,\dots, \gamma_t\}$ of $\bN^2$ and polynomials 
$W,V\in\bC[C_{\gamma_1},\dots,C_{\gamma_t}]$ such that 
$Q(a_{i_1},\dots,a_{i_s})=
W(c_{\gamma_1}/c_{0n},\dots,c_{\gamma_t}/c_{0n})/
V(c_{\gamma_1}/c_{0n},\dots,c_{\gamma_t}/c_{0n})$. Moreover 
$V(c_{\gamma_1}/c_{0n},\dots,c_{\gamma_t}/c_{0n})\neq0$. 
Since for some positive integer $k$, the product $c_{0n}^k\cdot W(c_{\gamma_1}/c_{0n},\dots,c_{\gamma_t}/c_{0n})$ is a polynomial 
in $c_{0n},c_{\gamma_1},\dots,c_{\gamma_t}$, we get the theorem. 
\end{proof}

For  a polynomial $Q$ from the above theorem, the condition $Q(a_{i_1},\dots,a_{i_s})\ne 0$ is called a {\it non-degeneracy condition} for the Pusieux parametrisation of $f$. Theorem \ref{t3} shows that such a condition can by expressed by a polynomial depending on the coefficients of the power series $f$.

\medskip
\noindent
{\small   Beata Gryszka\\
Institute of Mathematics\\
Pedagogical University of Cracow\\
Podchor\c{a}\.{z}ych 2\\
PL-30-084 Cracow, Poland\\
e-mail: bhejmej1f@gmail.com}

\medskip
\noindent
{\small   Janusz Gwo\'zdziewicz\\
Institute of Mathematics\\
Pedagogical University of Cracow\\
Podchor\c{a}\.{z}ych 2\\
PL-30-084 Cracow, Poland\\
e-mail: janusz.gwozdziewicz@up.krakow.pl}

\end{document}